\documentclass[a4paper,12pt,reqno]{amsart}

\usepackage{hyperref}
\usepackage{bookmark}
\hypersetup{
	pdfstartview=FitH,
	colorlinks=true,
	linkcolor=blue,
	citecolor=red,
}
\usepackage[dvipsnames]{xcolor}

\usepackage{amsfonts,amsmath,amssymb,amsfonts,amsthm}
\usepackage[left=3cm,right=2.6cm,top=2.7cm,bottom=2.5cm,twoside]{geometry}
\numberwithin{equation}{section}

\usepackage{times}

{
\newtheorem{theorem}{Theorem}[section]

\newtheorem{lemma}[theorem]{Lemma}

}

\newtheorem{conjecture}{Conjecture}

\newtheorem{thm}{Theorem}

{\theoremstyle{definition}
	}

\title[On the periodicity of entire functions]{On a theorem of A. and C. R\'enyi and a conjecture of C. C. Yang concerning periodicity of entire functions}
\author{Z.~Latreuch and M.~A.~Zemirni$^*$}
\thanks{$^*$ Corresponding author.}
\date{}

\address[Z.~Latreuch]{University of Mostaganem\\
	Department of Mathematics\\
	Laboratory of Pure and Applied Mathematics\\
	B.~P.~227 Mostaganem, Algeria.}
\email{z.latreuch@gmail.com}

\address[M.~A.~Zemirni]{University of Eastern Finland\\
	Department of Physics and Mathematics\\
	P.O. Box 111, 80101 Joensuu, Finland.}
\email{amine.zemirni@uef.fi}

\subjclass[2020]{30D20, 30D35.}
\keywords{Differential polynomials, entire functions, Nevanlinna theory, order of growth, periodic functions, Yang's conjecture.}

\begin{document}
\maketitle


\begin{abstract}
	A theorem of A. and C. R\'enyi on periodic entire functions states that an entire function $f(z) $ must be periodic if $ P(f(z)) $ is periodic, where $ P(z) $ is a non-constant polynomial. By extending this theorem, we can answer some open questions related to the conjecture of C. C. Yang concerning periodicity of entire functions. Moreover, we give more general forms for this conjecture and we prove, in particular, that $ f(z) $ is periodic if either  $ P(f(z)) f^{(k)}(z) $ or $ P(f(z))/f^{(k)}(z) $ is periodic, provided that $ f(z) $ has a finite Picard exceptional value.  We also investigate the periodicity of $ f(z) $ when $ f(z)^{n}+a_{1} f^{\prime}(z)+\cdots+a_{k} f^{(k)}(z) $ is periodic. In all our results, the possibilities for the period of $ f(z) $ are determined precisely.
\end{abstract}

\section{Introduction}

Periodicity of entire functions is associated with numerous difficult problems despite its simple concept. It has been studied from different aspects, such as uniqueness theory, composite functions, differential and functional equations; 
%
see \cite{L} and references therein. 
In this paper, special attention is paid to the problem of \emph{the periodicity of entire functions $ f(z) $ when particular differential polynomials generated by $ f(z) $ are given to be periodic}. By using concepts from Nevanlinna theory  (see, e.g., \cite{Laine,YY}), we can extend the following result due to Alfr\'ed and Catherine R\'enyi.
		\begin{thm}[{\cite[Theorem~2]{RR}}]\label{thm_RR}
			Let $Q(z)$ be an non-constant polynomial and $f(z)$ be an entire function. If $Q(f(z))$ is a periodic function, then $f(z)$ must be periodic.
		\end{thm}

With the help of the extensions of Theorem~\ref{thm_RR} together with other results from Nevanlinna theory, we study the problem mentioned above and answer some related open questions. In particular, we are interested in the following conjecture and its variations.

\begin{conjecture}[Generalized Yang's conjecture]\label{con}
	Let $f(z)$ be a transcendental entire function and $n,k$ be positive integers. If $f(z)^n f^{(k)}(z)$ is a periodic function, then $f(z)$ is also a periodic function.
\end{conjecture}

This conjecture is originally stated in \cite{LLY} for $ n=1 $  and has been known as Conjecture of C. C. Yang. The present form of Conjecture~\ref{con} is stated in \cite{LWY}.
The case $k=1$ of Conjecture~\ref{con} was settled in \cite[Theorem~1]{WH} for $ n=1 $, and in \cite[Theorem~1.2]{LWY} for $ n\ge 2 $.  

\medskip
In the following, we will present three results related to Conjecture~\ref{con}, which are the main subject of this paper. Before that, we recall that the order and the hyper-order of an entire function $ f(z) $ are defined, respectively, by 
	\begin{equation*}
	\rho(f)=\limsup _{r \rightarrow \infty} \frac{\log T(r, f)}{\log r}, \quad \rho_2(f)=\limsup _{r \rightarrow \infty} \frac{\log\log T(r, f)}{\log r},
	\end{equation*}
where $ T(r,f) $ is the Nevanlinna characteristic function of $ f(z) $. 

\bigskip 

Regarding the case when $ k\ge 2 $ in Conjecture~\ref{con}, we mention the following result.

\begin{thm}[{\cite[Theorem 1.2]{LWY}}]\label{thm_LWY}
	Let $ f(z) $ be a transcendental entire function and  $ n ,k$ be positive integers. Suppose that $ f(z)^n f^{(k)}(z) $ is a periodic function with period $ c $, $ f(z) $ has a finite Picard exceptional value, and that $ \rho(f)<\infty $, then $ f(z) $ is a periodic function of period $ c $ or $ (n+1) c $.
\end{thm}

If $ n=1 $ and $ f(z) $ has a non-zero Picard exceptional value, then \cite[Theorem 1.1]{LY} reveals that the conclusion of Theorem~\ref{thm_LWY} occurs without condition on the order. The case when $ n=1 $ and $ f(z) $ has $ 0 $ as a Picard exceptional value is proved in  \cite[p.~455]{LK}. Recently, Theorem~\ref{thm_LWY} was improved in \cite[Theorem 1.1]{LZ} to hold for entire functions with $ \rho_2(f)<1 $ having a finite Borel exceptional value.

%

\medskip
The idea of allowing $ n $ to be negative integer in Conjecture~\ref{con} is considered in \cite{LKL}, where the following particular result is obtained.
%
%
%
%
	\begin{thm}[{\cite[Theorem~1.5]{LKL}}]\label{thm_C}
		Let $ f(z) $ be a transcendental entire and $ k $ be a positive integer. Suppose that $f^{(k)}(z)/f(z)$ is a periodic function with period $ c $, and $ f(z) $ has a Picard exceptional value $ d \not= 0 $. Then $ f(z) $ is a periodic function with period $ c $.
	\end{thm}
	
	The conclusion of Theorem~\ref{thm_C} is not always true when $ d=0 $, as shown in \cite{LKL} by taking $ f(z) = e^{e^{i z}+z} $.   

\medskip
Another variation of Conjecture~\ref{con} is addressed in \cite{LZ}, where the differential polynomial $ f(z)^n f^{(k)}(z)$ is replaced with $f(z)^{n}+a_{1} f^{\prime}(z)+\cdots+a_{k} f^{(k)}(z)$. 
In this regard, we present the following result, which combines \cite[Theorem 1.2, Remark~1.2(ii)]{LZ}.

\begin{thm}\label{lu-zhang}
	Let $f(z)$ be a transcendental entire function, and let $ k\ge 1 $ be an integer and $ a_1, \ldots, a_k $  are constants. If $f(z)^{n}+a_{1} f^{\prime}(z)+\cdots+a_{k} f^{(k)}(z)$ is a periodic function with period $ c $, and if one of the following conditions holds
	\begin{enumerate}
		\item $ n=2 $ or $ n\ge 4 $,
		\item  $ n=3 $ and $ \rho_2(f)<1 $,
	\end{enumerate}
	then $ f(z) $ is periodic of period $ c $ or $ nc $.
\end{thm}

%

\medskip
The rest of the paper is devoted to improving Theorems~\ref{thm_RR}--\ref{lu-zhang}, and it is organized as follows. We prove two extensions of Theorem~\ref{thm_RR} in Section~\ref{RR}, which will be used to prove the results of Section~\ref{YC}. In Section~\ref{YC}, we improve and generalize Theorems~\ref{thm_LWY} and \ref{thm_C}.  Meanwhile, Section~\ref{VYC} contains results improving the Case (2) in Theorem~\ref{lu-zhang}.


 \section{On a Theorem of A. and C. R\'enyi}\label{RR}

Let $ f(z) $ be an entire function. The notation $ S(r,f) $ stands for any quantity satisfying $ S(r,f) = o(T(r,f)) $ as $ r\to \infty $ possibly outside an exceptional set of finite linear/logarithmic measure.  A meromorphic function $ g(z) $ is said to be small function of $ f(z) $ if and only if $ T(r,g)= S(r,f) $.

\medskip

The question addressed here is whether the conclusion of Theorem~\ref{thm_RR} can still hold if the polynomial in $ f(z) $,  $ Q(f) $,  is replaced with a rational function in $ f(z) $, $ P_1(z,f)/P_2(z,f) $, with coefficients being small functions of $ f(z) $. In this section, we treat this question for particular rational functions, and we discuss the sufficient conditions that we offered.  Before stating our results, we recall some notations. First, we introduce a quantity $ \delta_P $ assigned to a polynomial 
	$$ 
	P(z) =c_{\nu_1}z^{\nu_1} + \cdots + c_{\nu_\ell}z^{\nu_\ell}, \quad  \ell \ge 2, \; \nu_1< \cdots < \nu_\ell,
	$$ 
where $ c_{\nu_s}\neq 0  $ for all $ s=1, \ldots,\ell $, and defined by $ \delta_P := \gcd\left(\nu_2-\nu_1, \ldots, \nu_\ell - \nu_{\ell-1} \right)$. For example, if $ P(z) = z^n + z^{m+n} $ ($ m>0 $), then $ \delta_P = m$. Notice that when $ \nu_1 =0 $, then $ \delta_P = \gcd(\nu_2, \ldots, \nu_\ell) $.   The functions $ N(r,f) $ and $ N(r,1/f) $ are, respectively, the integrated counting functions for the poles and zeros of $ f(z) $ in the disc $|z|\le r $ with counting the multiplicities. Analogously, the functions $ \overline{N}(r,f) $ and $ \overline{N}(r,1/f) $ are, respectively, the integrated counting functions for the poles and zeros of $ f(z) $ in the disc $ |z|\le r $ ignoring the multiplicities.

\medskip
We proceed now to state and prove our results. 

\begin{theorem}\label{lemma1}
Let $ g(z) $ be a transcendental entire function such that $N(r,1/g)=S(r,g)$, $ A(z) $ be non-vanishing meromorphic function satisfying $ T(r,A) = S(r,g) $, and let $ Q(z) $ be a polynomial with at least two non-zero terms. If $ A(z) Q\left(g(z)\right)$ is periodic of period $ c $, then $ g(z) $ is periodic of period $ c $ or $ \delta_Q c $.
\end{theorem}

From Theorem~\ref{thm_RR}, we see that the conditions on $ Q(z) $ and on zeros of $ g(z) $ can be dropped in Theorem~\ref{lemma1} if $ A(z) $ is constant. In case of $ Q(z) $ is a monomial and $ A(z) $ is non-constant,  the conclusion of Theorem~\ref{lemma1} does not hold in general. In fact, we can always find functions $ A(z) $ and $ g(z) $ for which $ A(z) g(z)^n$ is periodic, but $ g(z) $ is not periodic, where $ n $ is a positive integer. This can be seen, for example, by taking $ A(z)= e^{-nz} $ and $ g(z) = e^{e^{iz}+z} $.   

\medskip
The occurrence of the possibility when $ g(z) $ is of period $ \delta_{Q} c $ can be seen by taking $g(z)=e^{z/3}  $ and $ Q(z) = z^3 + z^6 + z^9 $, where $ \delta_{Q}=3 $. In this case, $ Q(g(z)) = e^z + e^{2z} + e^{3z} $ is of period $ c= 2\pi i$ while $ g(z) $ is of period $ 3c=6\pi i $.

\medskip
The following result is a modification of Theorem~\ref{lemma1}, which treats the case when irreducible rational function in $ f(z) $ is periodic.

\begin{theorem}\label{lemma2}
	Let $ g(z) $ be a transcendental entire function such that $N(r,1/g)=S(r,g)$, $ A(z) $ be non-vanishing meromorphic function satisfying $ T(r,A) = S(r,g) $, and let $ Q(z) $ be a polynomial with at least two non-zero terms. Suppose that $ A(z) Q\left(g(z)\right)/g(z)$ is periodic of period $ c $. Then  
	$ g(z) $ is periodic of period $ 2c $ if $ Q(z)=c_0+c_1z+c_2z^2 $, where $ c_0c_2\neq 0 $. Otherwise, $ g(z) $ is periodic of period $ c $ or $ \delta_{Q} c $.
\end{theorem}

The case when $ Q(z) $ has the form $ Q(z)=c_0+c_1z+c_2z^2  $ in Theorem~\ref{lemma2} and $ g(z) $ is of period $ 2c $ may occur. For example, the function $ g(z)=e^{e^z} $ is of period $ 2\pi i $ while $ Q(g(z))/g(z) $ is of period $ \pi i $, where $ Q(z) = 1+c z+z^2 $ and $ c\in\mathbb{C} $.

\bigskip 
In the proofs Theorems~\ref{lemma1} and \ref{lemma2},  we will use $ g_c $ and $ A_c$ to stand for $ g(z+c) $ and $ A(z+c) $, respectively. Also, we will frequently use the following lemma.

\begin{lemma}[{\cite[Theorem~1.62]{YY}}]\label{lemY} 
	Let $ f_1(z), \ldots, f_n(z) $, where $ n\ge 3 $, be meromorphic functions which are non-constant except probably $ f_n(z) $. Suppose that $ \sum\limits_{j=1}^n f_j(z) =1 $ and $ f_n(z) \not\equiv 0 $. If
	$$
	\sum_{i=1}^{n}N\left( r,\frac{1}{f_i}\right) +(n-1)\sum_{\substack{i=1\\ i\neq j}}^{n}\overline{N}(r,f_i)<(\lambda+o(1))T(r,f_j), \quad  j=1,2,\ldots,n-1,
	$$
	as $ r\to\infty$ and $ r\in I$, where $\lambda<1$ and $I\subset [0,\infty)  $ is a set of infinite linear measure, then $ f_n(z) \equiv 1 $.
\end{lemma}

\begin{proof}[Proof of Theorem~\ref{lemma1}]
Let $ Q(z) = \sum_{s=1}^{\ell} c_{\nu_s} z^{\nu_s} $, where $ \ell\ge 2 $, $ \nu_1< \ldots< \nu_\ell$ and $ c_{\nu_s} \neq 0$ for all $ s=1, \ldots,\ell $.  Since~$ A(z) Q\left(g(z)\right)$ is periodic of period $ c $, it follows that	
	\begin{equation}\label{lem1}
	A_c \sum_{s=1}^{\ell} c_{\nu_s} g^{\nu_{s}}_c = A \sum_{s=1}^{\ell} c_{\nu_s} g^{\nu_{s}}.
	\end{equation}
Since $ T(r,A) = S(r,g) $, it follows from \eqref{lem1} that $ T(r,g_c) \sim  T(r,g)$ as $ r\to\infty $ probably outside an exceptional set of finite linear measure. 
Let $ m\in \{1, \ldots, \ell\} $. Dividing \eqref{lem1} by $c_{\nu_m} A(z)  g(z)^{{\nu_m}} $ results in
	\begin{equation}\label{lem2}
	\frac{A_c}{A} \sum_{s=1}^{\ell} \frac{c_{\nu_s}}{c_{\nu_m}} \frac{g^{\nu_{s}}_c}{g^{\nu_m}} -  \sum_{\substack{s=1\\ s\neq m}}^{\ell} \frac{c_{\nu_s}}{c_{\nu_m}} g^{{\nu_s}-{\nu_m}}=1.
	\end{equation}
Clearly $ \frac{A_c}{A} \frac{g^{\nu_{s}}_c}{g^{\nu_m}}$ is non-constant for all $ s\neq m $. Then by Lemma~\ref{lemY}, we obtain
	\begin{equation*}
	\left( \frac{g_c}{g}\right) ^{\nu_m } = \frac{A}{A_c}.
	\end{equation*}
Since this is true for every $ m=1, \ldots, \ell $, it follows that
	$$
	\left( \frac{g_c}{g}\right) ^{\nu_2-\nu_1}=\cdots=\left( \frac{g_c}{g}\right) 	^{\nu_\ell-\nu_{\ell-1}}=1.
	$$
This means $ g(z) $ is periodic of period $ c $ or $ \delta_Q c$.	
\end{proof}

\begin{proof}[Proof of Theorem~\ref{lemma2}]
	Let $ Q(z)$ be as in Proof of Theorem~\ref{lemma1}. If $  \nu_1 \neq 0$, then $ Q(z) z^{-1} $ is a polynomial, and hence the results follows from Theorem~\ref{lemma1}. So, we assume that $ \nu_1 =0 $. Similarly as in the proof of Theorem~\ref{lemma1}, one can see that  $ T(r,g_c) \sim  T(r,g)$ as $ r\to\infty $ probably outside an exceptional set of finite linear measure.
	 Next we distinguish two cases:
	
	(1) Case $ \nu_2 =1 $:  By periodicity of $ A(z) Q\left(f(z)\right)/f(z)$, we have
		\begin{equation}\label{e_lem_1}
		\frac{A_{c} c_{0}}{g_c} +A_{c} c_{1}+A_{c} \sum_{s=3}^{\ell} c_{\nu_{s}} g_c^{\nu_{s}-1 	}-\frac{A c_{0}}{g} -A c_{1}-A \sum_{s=3}^{\ell} c_{\nu_{s}} g^{\nu_{s}-1 }=0.
		\end{equation}
	Dividing \eqref{e_lem_1} by $c_1 A(z) $ and using Lemma~\ref{lemY} yield  $ A(z+c) \equiv A(z) $. Using this, and multiplying \eqref{e_lem_1} by $ g(z) / c_0 $, we find
		\begin{equation}\label{e_lem_2}
		 \frac{g}{g_c}+\sum_{s=3}^{\ell} \frac{c_{\nu_{s}}}{c_0} gg_c^{\nu_{s}-1 }-\sum_{s=3}^{\ell} 	\frac{c_{\nu_{s}}}{c_0} g^{\nu_{s} }=1.
		\end{equation}
	If $ \ell =2 $, then the sums in \eqref{e_lem_2} are vanished, and therefore $ g(z) $ is periodic of period $ c $.	
	 We suppose then that $ \ell \ge 3 $. Notice that $ gg_c^{\nu_{s}-1 } $ is non-constant unless  $ \nu_3=2 $. Also, notice that $ {g}/{g_c} $ and $ gg_c $ cannot be constants simultaneously. Hence, if $ gg_c $ is non-constant or $ \nu_3>2 $, then by applying Lemma~\ref{lemY} on \eqref{e_lem_2}, we obtain that $ g(z) $ is periodic of period $ c $. Next, we suppose that $ \nu_3=2 $ and $ {g}/{g_c} $ is non-constant. Then again from Lemma~\ref{lemY} and \eqref{e_lem_2} we obtain that
	 	\begin{equation}\label{pro}
	 	gg_c=c_0/c_2,
	 	\end{equation}
	 which implies that $ g(z) $ is periodic of period $ 2c $. We show next that this case can occur only if $ \ell=3 $. Assume that $ \ell \ge 4 $. We obtain from \eqref{pro} that $ g(z)=e^{\alpha(z)} $ and $ g(z+c)=(c_0/c_2) e^{-\alpha(z)} $, where $ \alpha(z) $ is an entire function.   Then from \eqref{e_lem_2} we obtain 
	 		\begin{equation*}
	 	\sum_{s=4}^{\ell} \frac{c_{\nu_{s}}}{c_0}  (c_0/c_2)^{\nu_{s}-1}e^{-(\nu_{s}-2)\alpha(z)}- \sum_{s=4}^{\ell} 	\frac{c_{\nu_{s}}}{c_0} e^{\nu_{s}\alpha(z)}=0,
 			\end{equation*}
	 which is not possible by Borel's lemma. Thus $ \ell =3 $, and therefore $ Q(z)=c_0+c_1 z +c_2 z^2  $.

	(2) Case $ \nu_2\ge 2 $: In this case, we have
		\begin{equation}\label{e_lem_4}
		\frac{A_{c} c_{0}}{g_c} +A_{c} \sum_{s=2}^{\ell} c_{\nu_{s}} g_c^{\nu_{s}-1 }-\frac{A 	c_{0}}{g} -A \sum_{s=2}^{\ell} c_{\nu_{s}} g^{\nu_{s}-1}=0.
		\end{equation}
	Dividing \eqref{e_lem_4} by $ \frac{Ac_0}{g} $ results in
		\begin{equation}\label{e_lem_5}
			\frac{A_{c} g}{A g_c} +\frac{A_{c}}{A} \sum_{s=2}^{\ell} \frac{c_{\nu_{s}}}{c_0} gg_c^{\nu_{s}-1 } - \sum_{s=2}^{\ell} \frac{c_{\nu_{s}}}{c_0} g^{\nu_{s}}=1.
		\end{equation}
	Here we notice that $ \frac{A_{c}}{A} gg_c^{\nu_{s}-1 } $ is non-constant unless $ \nu_2=2 $. Also, $ \frac{A_{c} g}{A g_c} $ and $ \frac{A_{c}}{A} gg_c $ cannot be constants simultaneously.  If $ \frac{A_{c}}{A} gg_c $ is non-constant or $ \nu_2>2 $, then from \eqref{e_lem_5} and Lemma~\ref{lemY} we obtain that $A_c/A = g_c/g $. Then replacing $ A_c/A $ with $g_c/g $ in \eqref{e_lem_4} results~in 
		\begin{equation*}
		\sum_{s=2}^{\ell} c_{\nu_{s}} g_c^{\nu_{s}}-\sum_{s=2}^{\ell} c_{\nu_{s}} g^{\nu_{s} }=0.
		\end{equation*}
	By Theorem~\ref{lemma1}, we deduce that $ g(z) $ is periodic of $ c $ or $ \delta_Q c $.  
	
	Now,  if $ \frac{A_{c} g}{A g_c} $ is non-constant and $ \nu_2=2 $, then from \eqref{e_lem_5} and Lemma~\ref{lemY} we obtain
		\begin{equation}\label{pro1}
			  gg_c =\frac{c_0}{c_2} \frac{A}{A_{c}}.
		\end{equation}
	On the other hand, dividing \eqref{e_lem_4} by $ \frac{A_cc_0}{g_c} $ results in
		\begin{equation}\label{e_lem_6}
			 -\sum_{s=2}^{\ell} \frac{c_{\nu_{s}}}{c_0} g_c^{\nu_{s} }+\frac{A 	g_c}{A_c g} +\frac{A}{A_c} \sum_{s=2}^{\ell} \frac{c_{\nu_{s}}}{c_0} g_c g^{\nu_{s}-1}=1.
		\end{equation}
	Hence from \eqref{e_lem_6} and Lemma~\ref{lemY} we obtain
		\begin{equation}\label{pro2}
			 g_c g =\frac{c_0}{c_2}\frac{A_c}{A}.
		\end{equation}
	From \eqref{pro1} and \eqref{pro2}, it follows that $ g_c g = \pm c_0/c_2 $. Similarly as previous case (1), we deduce that $ g(z) $ is periodic of period $ 2c $ and $ Q(z) $ has, in this case, the form $ Q(z)=c_0+c_2z^2 $.
\end{proof}

%

\section{Results on Yang's conjecture} \label{YC}

We start this section by proving the following result, which improves Theorem~\ref{thm_LWY}.

\begin{theorem}\label{thm_1}
	Let $ f(z) $ be a transcendental entire function and  $ n ,k$ be positive integers. Suppose that $ f(z)^n f^{(k)}(z) $ is a periodic function with period $ c $, and one of the following  holds.
	\begin{enumerate}
		\item[(i)] $ f(z) $ has the value $ 0 $ as Picard exceptional value, and $ \rho_2(f)<\infty $.
		\item[(ii)] $ f(z) $ has a non-zero Picard exceptional value.
	\end{enumerate}
	Then $f(z)$ is a periodic function of period $c$ or $(n+1)c$.
\end{theorem}

\begin{proof}
	By the assumption,  $ f(z) $ has the form $ f(z)=e^{h(z)}+d $, where $ h(z) $ is an entire function, and $ d\in\mathbb{C} $. 
		
	From Theorem~\ref{thm_LWY}, we need only to consider the case when $h(z)$ is a transcendental entire function. 

	(i) Case $ d=0 $:  Since $f(z)^nf^{(k)}(z)$ is a periodic function of period $c$, it follows
	\begin{equation*}\label{eq_c}
		f(z)^nf^{(k)}(z)=f(z+c)^nf^{(k)}(z+c).
	\end{equation*}
	Substituting $f(z)=e^{h(z)}$, where $ \rho(h)<\infty $, into this equation gives
	\begin{eqnarray}\label{1}
		e^{(n+1)\Delta h(z)}=\frac{B_k(z)}{B_k(z+c)},
	\end{eqnarray}
	where $ \Delta h(z) = h(z+c)-h(z) $ and 
	\begin{equation}\label{B_k}
		B_k(z) = (h'(z))^k + Q_{k-1}(z),
	\end{equation}  
	and $ Q_{k-1}(z) $ is a differential polynomial in $ h'$ with constant coefficients and of degree $k-1 $.

	Assume that $\Delta h(z)$ is a transcendental entire function. Then \eqref{1} yields	
	$$
	\infty=\rho\left(e^{(n+1)\Delta h}\right)= \rho\left(\frac{B_k(z)}{B_k(z+c)}\right)\leq \rho(h)<\infty,
	$$
	which is a contradiction. Thus $ \Delta h(z) $ is a polynomial.  Now, assume that $ \Delta h(z)$  is non-constant polynomial of degree $p\geq 1$. Then \eqref{1} can be seen as a linear difference equation
	\begin{eqnarray}\label{1_1}
		B_k(z+c)=e^{-(n+1)\Delta h(z)}\, B_k(z).
	\end{eqnarray}
	From this and \cite[Theorem~9.2]{CF}, we obtain that $\rho(B_k)\geq \rho\left(e^{(n+1)\Delta h}\right) +1= p+1$. Hence 
	\begin{equation}\label{1_2}
		\rho(h') \ge \rho(B_k)\geq p+1.
	\end{equation}
	Since $ \Delta h(z) $ is polynomial, it follows from \eqref{B_k} and \eqref{1_1} that
	\begin{eqnarray}\label{1_3}
		(h'(z))^k=-Q_{k-1}(z) +\frac{L_{k-1}(z)}{1-e^{-(n+1)\Delta h(z)}},
	\end{eqnarray}
	where $L_{k-1}(z)$ is is a differential polynomial in $h'$ with polynomial coefficients of degree $k-1$.
	Therefore, \eqref{1_3} 
	results in
	$$
	kT(r,h')\leq (k-1)T(r,h')+O(r^p)+O(\log r),
	$$
	and hence $\rho(h')\leq p$, which contradicts \eqref{1_2}. Thus 	$\Delta h(z)$ is a constant. In this case, one can easily see from \eqref{B_k} that $B_k(z+c)=B_k(z)$. Thus $e^{(n+1)(h(z+c)-h(z))}=1$, that is, $f(z)$ is a periodic function of period $c$ or $(n+1)c$.

	(ii) Case $ d\neq 0 $: In this case, we have
	\begin{equation}\label{e_0}
		f(z)^nf^{(k)}(z) = B_k(z) \sum_{j=0}^n  {n \choose j} d^{n-j} e^{(j+1)h(z)}, 
	\end{equation}
	where  $ d\neq 0 $ and $ B_k(z) $ is defined as in \eqref{B_k}. From Theorem~\ref{lemma1}, we obtain that $ f(z) $ is periodic of period $ c $. This completes the proof.
\end{proof}

In the following result, we extend Theorem~\ref{thm_C} to hold for $f^{(k)}(z)/f(z)^n$.

\begin{theorem}\label{thm_1_1}
	Let $ f(z) $ be a transcendental entire function, and let  $ k\ge 1$ and $ n\ge 2 $ be integers. Suppose that $ f^{(k)}(z)/f(z)^n $ is a periodic function with period $ c $, and one of the following  holds.
	\begin{enumerate}
		\item[(i)] $ f(z) $ has the value $ 0 $ as Picard exceptional value, and $ \rho_2(f)<\infty $.
		\item[(ii)] $ f(z) $ has a non-zero Picard exceptional value.
	\end{enumerate}
	Then $f(z)$ is a periodic function of period $c$, $ 2c $ or $(n-1)c$.
\end{theorem}

\begin{proof}
	Since $ f^{(k)}(z)/f(z)^n $ is periodic with period $ c $, it follows that 
	\begin{equation}\label{eq1}
		\frac{f(z+c)^n}{f^{(k)}(z+c)} = \frac{f(z)^n}{f^{(k)}(z)}. 
	\end{equation}
	By substituting the form $ f(z) = e^{h(z)} +d$ in \eqref{eq1}, we obtain
	\begin{equation}\label{eq2}
		\frac{1}{B_k(z+c)}\left(e^{h(z+c)} +d\right)^n e^{-h(z+c)} =\frac{1}{B_k(z)} \left(e^{h(z)} +d\right)^n e^{-h(z)},
	\end{equation}
	where $ B_k(z) $ is defined in \eqref{B_k}.  Next, we distinguish the cases related to $ d $.
	
	(i) Case $ d=0 $:  In this case,  \eqref{eq2} becomes
	$$
	e^{(n-1) \Delta h(z)} = \frac{B_k(z+c)}{B_k(z)}.
	$$
	Then, from proof of Theorem~\ref{thm_LWY} when $ h(z) $ is polynomial, and from proof Theorem~\ref{thm_1}(i) when $ h(z) $ is transcendental, we obtain that $ f(z) $ is a periodic function with period $ c $ or $ (n-1)c $.
	
	(ii) Case $ d\neq 0 $:   In this case, the result follows directly from Theorem~\ref{lemma2}. This completes the proof.
\end{proof}

If we replace the monomial $ f(z)^n $ by a polynomial with at least two non-zero terms in Theorem~\ref{thm_1} and Theorem~\ref{thm_1_1}, then the restriction on the growth in Case (i) in both theorems is no more needed. 

\begin{theorem}\label{thm_1_poly}
	Let $ f(z) $ be a transcendental entire function, $ P(z) $ be a polynomial with at least two non-zero terms, and $ k \ge 1$  be an integer. Suppose that $ P(f(z)) f^{(k)}(z) $ is a periodic function with period $ c $, and $ f(z) $ has a finite Picard exceptional value. Then $ f(z) $ is a periodic function of period $c$ or $ \delta_P c $. 
\end{theorem}
\begin{proof}
	Let $ P(z) = \sum_{j=1}^{n} b_{\nu_j} z^{\nu_j} $. 
	
	(i) Case $ d=0 $: In this case we have
	\begin{equation}\label{ep2}
		P(f(z)) f^{(k)}(z) =B_k(z)\sum_{j=1} ^n b_{\nu_j} e^{(\nu_j+1)h(z)},
	\end{equation}
	where $ B_k(z) $ is defined in \eqref{B_k}. From Theorem~\ref{lemma1}, we deduce that $ f(z) $ is periodic with period $ c $ or $ \delta_P c $.
	
	(ii) Case $ d\neq 0 $: In this case, we have
	$$
	P(f(z)) f^{(k)} (z) = B_k(z)  \sum_{j=1}^n b_{\nu_j} \left(\sum_{s=0}^{\nu_j} a_s e^{(s+1)h(z)}\right),
	$$
	where $ a_s = {\nu_j \choose s} d^{\nu_j-s}$. By changing the order of the sums in the previous equality, we get
	$$
	P(f(z)) f^{(k)} (z) = B_k(z)\sum_{s=0}^{\nu_n}  c_s e^{(s+1)h(z)}, \quad   c_s = a_s\sum_{j=s}^{\nu_n} b_j.
	$$
	From Theorem~\ref{lemma1}, we conclude that $ f(z) $ is a periodic function with period $ c $.
\end{proof}
Theorem~\ref{thm_1_poly} generalizes and improves \cite[Corollary~1.5 and Theorem~1.6]{WLL} (see also \cite[Corollary ~5.1 and Theorem~5.14]{L}).

\begin{theorem}\label{thm_1_1_poly}
	Let $ f(z) $ be a transcendental entire function, $ P(z) $ be a polynomial with at least two non-zero terms, and $ k \ge 1$  be an integer. Suppose that $ f^{(k)}(z) / P(f(z))$ is a periodic function with period $ c $, and $ f(z) $ has a finite Picard exceptional value. Then $ f(z) $ is a periodic function of period $c$, $ 2c $ or $ \delta_P c $. 
\end{theorem}
\begin{proof}
	We follow the proof of Theorem~\ref{thm_1_poly} by using Theorem~\ref{lemma2}.
\end{proof}
Theorem~\ref{thm_1_1_poly} generalizes and improves \cite[Theorem~1.7]{WLL} (see also \cite[Theorem~5.16]{L}).


\section{On a variation of Yang's conjecture} \label{VYC}

Recall that the $ p $-iterated order of an entire function $ f(z) $ is defined by 
	\begin{equation*}
	\rho_p(f)=\limsup _{r \rightarrow \infty} \frac{\log_p T(r, f)}{\log r},
	\end{equation*}
where $ \log_1 r= \log r$ and $ \log_{p} = \log\left( \log_{p-1} r\right) $. The {finiteness degree} $ i(f) $ of  $ f(z) $ is defined to be $ i(f):=0 $ if $ f(z) $ is a polynomial, $i(f):= \min \left\{j \in \mathbb{N}: \rho_{j}(f)<\infty\right\} $ if there exists some $ j \in \mathbb{N} $ for which $ \rho_{j}(f)<\infty $,  or otherwise $ i(f):=\infty $.

\medskip
We find that the condition on the growth in Case (2) of Theorem~\ref{lu-zhang} can be relaxed to $ i(f)<\infty $.

\begin{theorem}\label{thm_3}
	Let $f(z)$ be a transcendental entire function, and let $ k\ge 1 $ be an integer and $ a_1, \ldots, a_k $  are constants. If $f(z)^3+a_{1} f^{\prime}(z)+\cdots+a_{k} f^{(k)}(z)$ is a periodic function with period $ c $ and $ i(f)<\infty $, then $ f(z) $ is periodic with period~$ c$, $2c$ or $3c$.
\end{theorem}

We find also that the condition $ i(f)<\infty $ can be dropped at the expense of allowing some restrictions on the zeros of $ f $. In fact, we replace the condition $ i(f)<\infty $ with $\Theta(0,f)>0 $, where
$$
\Theta(0,f):= 1-\limsup _{r \rightarrow \infty} \frac{\overline{N}\left(r, \frac{1}{f}\right)}{T(r, f)}.
$$

\begin{theorem}\label{thm_2}
	Let $f(z)$ be a transcendental entire function, and let $ k\ge 1 $ be an integer and $ a_1, \ldots, a_k $  are constants. If $f(z)^3+a_{1} f^{\prime}(z)+\cdots+a_{k} f^{(k)}(z)$ is a periodic function with period $ c $ and $\Theta(0,f)>0 $, then $ f(z) $ is periodic with period~$ c$ or $ 3c $.
\end{theorem}

To prove Theorems~\ref{thm_3} and \ref{thm_2}, we start first with general preparations.
 
Since $f(z)^3+a_{1} f^{\prime}(z)+\cdots+a_{k} f^{(k)}(z)$ is periodic with period $ c $, it follows
	\begin{equation}\label{e1}
	f(z+c)^3 - f(z)^3 = -\sum_{j=1}^{k} a_j[f^{(j)}(z+c)-f^{(j)}(z)].
	\end{equation}
If $ \sum_{j=1}^{k} a_j[f^{(j)}(z+c)-f^{(j)}(z)] \equiv 0 $, then $ f(z+c)^3 - f(z)^3 \equiv 0 $. Therefore, $ f(z) $ is periodic with period $ c $ or $ 3c $. Thus, in what follows we  suppose that  $ \sum_{j=1}^{k} a_j[f^{(j)}(z+c)-f^{(j)}(z)] \not\equiv 0 $. Moreover, we suppose that $ f(z+c)-f(z) \not\equiv 0 $.  Then \eqref{e1} can be written as
	\begin{equation}\label{e2}
	f(z+c)^2 +f(z+c)f(z)+ f(z)^2 = -\sum_{j=1}^{k} a_j \frac{f^{(j)}(z+c)-f^{(j)}(z)}{f(z+c)-f(z)}.
	\end{equation}
Let $ H(z) $ denotes the right hand side of \eqref{e2}. Then $ H(z) $ is an entire function satisfying
	\begin{equation}\label{e3}
	T(r,H) = O\left(\log T(r,\Delta f) + \log r\right), \quad r\notin E,
	\end{equation}
where $ \Delta f(z) = f(z+c)-f(z) $ and $ E\subset [0,\infty) $ is a subset of finite linear measure. 
By using the notation $ w(z) := f(z+c)/f(z) $,  \eqref{e2} is rewritten as
	\begin{equation}\label{e4}
	w^2 + w+1 = \frac{H}{f^2}.
	\end{equation}
By definition of $ \Delta f $, we have 
	\begin{equation}\label{e4_1}
	f= \frac{\Delta f}{w-1}.
	\end{equation} 
Substituting \eqref{e4_1} into \eqref{e4}, we get
	\begin{equation}\label{e4_2}
	\frac{w^2 + w+1}{(w-1)^2} =\frac{H}{(\Delta f)^2}.
	\end{equation}

Suppose now that $ w $ is constant. Then from \eqref{e3}, \eqref{e4} and \eqref{e4_1} we obtain 
	$$
	2T(r,f) = T(r,H) = O\left(\log T(r,\Delta f) + \log r\right) = O\left(\log T(r,f) + \log r\right),\quad r\not\in E,
	$$
which is a contradiction.  Thus we may suppose that $ w $ is non-constant in all what follows.

\begin{proof}[Proof of Theorem~\ref{thm_3}]
Since $ i(f)<\infty $, it follows that there exists some integer $ p \ge 1 $ for which $ \rho_p(f)<\infty $ and $ \rho_{p-1} (f) = \infty $.

From \eqref{e3}, \eqref{e4} and \eqref{e4_2}, we get
	\begin{equation*}
	T(r,f) = T(r,\Delta f) + O\left(\log T(r,\Delta f) + \log r\right), \quad r\notin E,
	\end{equation*}	
which results in $ i(\Delta f) =p $, and therefore $ i(H)\le p-1 $.
We rewrite \eqref{e2} as follows
	$$
	\left(f(z+c)-q f(z) \right) \left(f(z+c)-q^2 f(z)\right) = H(z),
	$$
where $ q = \frac{-1+i\sqrt{3}}{2} $. By Weierstrass factorization, we obtain
	\begin{equation}\label{equation1}
	f(z+c)-q f(z)=\Pi(z) \mathrm{e}^{\alpha(z)},
	\end{equation}
and then 
	\begin{equation}\label{equation2}
	f(z+c)-q^2 f(z)=\frac{H(z)}{\Pi(z)} \mathrm{e}^{-\alpha(z)},
	\end{equation}
where $ \Pi(z) $ is the canonical product, and $ \alpha(z) $ is an entire function.  From the properties of canonical products, it follows that $ \rho_{p-1}(\Pi)=\lambda_{p-1}(\Pi)\le \rho_{p-1}(H)<\infty $, see \cite[Satz 12.3]{JV}. If $ i\left(e^\alpha\right)\le p-1 $, then from \eqref{equation1} and \eqref{equation2} we have
	\begin{equation}\label{e9}
	(q^2-q)f(z) = \Pi(z) \mathrm{e}^{\alpha(z)}-\frac{H(z)}{\Pi(z)} \mathrm{e}^{-\alpha(z)},
	\end{equation}
and this yields $ i(f) \le p-1$, which is a contradiction. Thus $i\left(e^\alpha\right) = p$. In this case, by applying \cite[Corollary~4.5]{HLWZ}, we obtain that $ N(r,1/\Pi) \le N(r,1/H)\le T(r,H) = o(T(r,e^{\alpha})) $ as $r\to\infty$ and $ r\in F $, where $ \overline{\operatorname{logdens}}(F)=1 $.  
From \eqref{equation1} and \eqref{equation2}, we obtain
	\begin{eqnarray}\label{e14}
		\begin{split}		
	\frac{1}{q} \frac{H(z)}{\Pi(z)^2} \mathrm{e}^{-2 \alpha(z)}
	&+\frac{1}{q^{2}} \frac{\Pi(z+c)}{\Pi(z)} \mathrm{e}^{\alpha(z+c)-\alpha(z)} \\
	&-\frac{1}{q^{2}} \frac{H(z+c)}{\Pi(z)\Pi(z+c)} \mathrm{e}^{-\alpha(z+c)-\alpha(z)}=1 .
		\end{split}
	\end{eqnarray}
By using \cite[Theorem~1.56]{YY}, we distinguish two cases:
	\begin{itemize}
		\item [(i)] 
		
		We have:
			\begin{align*}
				-\frac{1}{q^{2}} \frac{H(z+c)}{\Pi(z)\Pi(z+c)} \mathrm{e}^{-\alpha(z+c)-\alpha(z)} =1
			\end{align*}
		and
			\begin{equation*}
			\frac{1}{q} \frac{H(z)}{\Pi(z)^2} \mathrm{e}^{-2 \alpha(z)}+\frac{1}{q^{2}} \frac{\Pi(z+c)}{\Pi(z)} \mathrm{e}^{\alpha(z+c)-\alpha(z)}=0,
			\end{equation*}
		which results in
			$$
			H(z)=-\frac{1}{q}e^{\alpha(z+c)+\alpha(z)} \Pi(z)\Pi(z+c)
			$$
		and
			$$
			H(z+c)=-q^{2} e^{\alpha(z+c)+\alpha(z)} \Pi(z)\Pi(z+c).
			$$
		Since $q^3=1$, it follows that $ H(z+c) = H(z) $ and $ \Pi(z+2c)e^{\alpha(z+2c)} = \Pi(z) e^{\alpha(z)}$. Therefore, \eqref{e9} reveals that $f(z)$ is a periodic function of period $2c$.

	\item [(ii)] We have
		\begin{equation*}\label{e15}
		\frac{1}{q^{2}} \frac{\Pi(z+c)}{\Pi(z)} \mathrm{e}^{\alpha(z+c)-\alpha(z)} \equiv 1 
		\end{equation*}
	and 
		\begin{equation*}
		\frac{1}{q} \frac{H(z)}{\Pi(z)^2} \mathrm{e}^{-2 \alpha(z)}-\frac{1}{q^{2}} \frac{H(z+c)}{\Pi(z)\Pi(z+c)} \mathrm{e}^{-\alpha(z+c)-\alpha(z)}\equiv 0,
		\end{equation*}
	that is 
		\begin{equation*}
		\Pi(z+c) \mathrm{e}^{\alpha(z+c)} = q^2 \Pi(z) \mathrm{e}^{\alpha(z)} 
		\end{equation*}
	and
		\begin{equation*}
		\frac{H(z+c)}{\Pi(z+c)} \mathrm{e}^{-\alpha(z+c)} = q \frac{H(z)}{\Pi(z)} \mathrm{e}^{-\alpha(z)}.
		\end{equation*}
	Since $q^3=1$, we can easily deduce from this and from \eqref{e9} that $ f(z) $ is periodic with period $ 3c $.\hfill$ \qedhere $
	\end{itemize}	
\end{proof}

\begin{proof}[Proof of Theorem~\ref{thm_2}]
Here, we prove that under the assumption that $f(z+c)-f(z)\not\equiv 0  $, we get a contradiction.
From \eqref{e3} and \eqref{e4_2}, we have $ S(r,w)=S(r, \Delta f) $.  Therefore,  \eqref{e4} leads~to
	\begin{equation}\label{e5}
	T(r,f) =  T(r,w) + S(r,w).
	\end{equation}
Using the second main theorem, together with \eqref{e4}, we obtain
	\begin{eqnarray*}
		T(r,w) &\le&  \overline{N}\left(r,w\right) + \overline{N}\left(r,\frac{1}{w-q}\right) +\overline{N}\left(r,\frac{1}{w-q^2}\right)+ S(r,w) \\
		&=& \overline{N}\left(r,w\right) + \overline{N}\left(r,\frac{1}{w^2+w+1}\right)+ S(r,w)\\
		&\le & \overline{N}\left(r,\frac1f\right) + \overline{N}\left(r,\frac{1}{H}\right)+ S(r,w) = \overline{N}\left(r,\frac1f\right) + S(r,w).
	\end{eqnarray*}
	Hence,  from \eqref{e5}, we obtain
	\begin{equation}\label{ine}
	T(r,f) \le  \overline{N}\left(r,\frac1f\right) + S(r,w).
	\end{equation}
	Since $ \Theta(0,f)>0 $, it follows for any $ \varepsilon \in (0,1) $ sufficiently small, there exists $ R>0 $ such that $ \overline{N}(r,1/f) < (1- \varepsilon) T(r,f) $. This together with \eqref{ine} results in $ T(r,f) = S(r,w) $, which contradicts \eqref{e5}. 
\end{proof}

\end{document}